\documentclass[11pt,letterpaper]{article}
\usepackage{amsfonts, amsmath, amssymb, amscd, amsthm, graphicx, mathrsfs, wasysym, mdwlist, setspace, color}
\usepackage{hyperref}
\makeatletter
\def\blfootnote{\xdef\@thefnmark{}\@footnotetext}
\makeatother

\newtheorem{thm}{Theorem}[section]
\newtheorem{cor}[thm]{Corollary}
\newtheorem{lem}[thm]{Lemma}
\newtheorem{prop}[thm]{Proposition}

\newtheorem{conj}[thm]{Conjecture}
\theoremstyle{definition}

\theoremstyle{remark}
\newtheorem{rem}[thm]{Remark}

\def\im{{\rm im}}
\def\dim{{\rm dim}}
\def\Z{\mathbb Z}

\def\C{\mathbb{C}}
\def\R{\mathbb{R}}

\def\N{\mathbb{N}}
\def\nrk{{\rm nrk}}

\newcommand{\G }{\Gamma (G, X\sqcup \mathcal H)}
\newcommand{\e }{\varepsilon }
\renewcommand{\kappa }{\varkappa}
\newcommand{\Hl }{\{ H_\lambda \mid \lambda \in \Lambda \} }
\newcommand{\lo }{\ell_\Omega }
\renewcommand{\ll }{\left\langle\hspace{-.7mm}\left\langle }
\newcommand{\rr }{\right\rangle\hspace{-.7mm}\right\rangle }
\renewcommand{\d }{{\rm d} }
\newcommand{\Lab }{{\bf Lab} }

\begin{document}

\onehalfspace

\title{Normal generation and $\ell^2$-Betti numbers of groups}
\author{Denis Osin\thanks{The research of the first author was supported by the NSF grant DMS-1006345 and by the RFBR grant 11-01-00945.} \ and Andreas Thom}
\date{}
\maketitle

\begin{abstract}
The \emph{normal rank} of a group is the minimal number of elements whose normal closure coincides with the group. We study the relation between the normal rank of a group and its first $\ell^2$-Betti number and conjecture the inequality $\beta_1^{(2)}(G) \leq \nrk(G)-1$ for torsion free groups.
The conjecture is proved for limits of left-orderable amenable groups.
On the other hand, for every $n\ge 2$ and every $\e>0$, we give an example of a simple group $Q$ (with torsion) such that $\beta_1^{(2)}(Q) \geq n-1-\varepsilon$. These groups also provide examples of simple groups of rank exactly $n$ for every $n\ge 2$; existence of such examples for $n> 3$ was unknown until now.
\end{abstract}

\section{Introduction}

Let $G$ be a countable group. By now, the first $\ell^2$-Betti number $\beta^{(2)}_1(G)$, a numerical invariant ranging in the interval $[0,\infty]$, has become a standard tool in the study of the properties of $G$. The study of this invariant was started by Atiyah \cite{atiyah}, extended and generalized by Cheeger-Gromov \cite{cheegro}, and further developed in important work of Bekka-Valette, L\"uck, and Gaboriau \cite{bv, gaboriau, MR1926649}. See \cite{MR1926649} for the necessary background and all relevant definitions.

It is well-known that for every finitely generated group $G$, we have
\begin{equation} \label{triv}
\beta_1^{(2)}(G) \leq \d(G)-1,
\end{equation}
where $\d (G)$ denotes the minimal number of generators of $G$. The proof of this statement is essentially trivial using the Morse inequality. The main goal of this note is to study the relation between the first $\ell^2$-Betti number and the normal rank of a group. Recall that a group $G$ is {\it normally generated} by a subset $X\subseteq G$ if $G$ coincides with the normal closure $\ll X \rr$ of $X$, i.e., the only normal subgroup of $G$ containing $X$ is $G$ itself. The \emph{normal rank} of $G$, denoted $\nrk (G)$, is the minimal number of normal generators of $G$.

In view of (\ref{triv}) it seems natural to ask whether the inequality
\begin{equation}\label{naive}
\beta_1^{(2)}(G) \leq \nrk(G)-1
\end{equation}
can hold in general. However, the group $PSL(2,\Z) = \langle a,b \mid a^2 = b^3=1 \rangle$ shows that this is too optimistic. Indeed, we have $\beta_1^{(2)}(PSL(2,\Z)) = 1/6$ whereas $\nrk(PSL(2,\Z))=1$, as $PSL(2,\Z)= \ll ab \rr$. Moreover, Lennox and Wiegold proved \cite{LW} that for any two finite perfect groups $A$ and $B$, the normal rank of the free product $A\ast B$ is one. The idea behind these examples can be generalized to show that there exist hyperbolic groups of normal rank one with arbitrary large first $\ell^2$-Betti numbers (see Remark \ref{hypgroups}). Moreover, using some standard tools provided by small cancellation theory over relatively hyperbolic groups we obtain the following.

\begin{thm}\label{main1}
For every integer $n\ge 2$ and every $\e >0$ there exists an infinite simple group $Q$ generated by $n$ elements such that $\beta_1^{(2)} (Q)\ge n-1-\e $.
\end{thm}

This theorem also has a purely group theoretic application to the following question asked by Wiegold: Is every finitely generated simple group generated by $2$ elements? The answer is positive for finite groups, Higman's simple groups, and many other examples. However it is negative in general. The first counterexample was constructed by Guba \cite{Gub} using geometric methods developed by Olshanskii \cite{Ols}. More precisely, Guba constructed a finitely generated simple group $S$ such that all $2$-generated subgroups of $S$ are free. In particular, $S$ is not $2$ generated. Guba showed that $S$ can be generated by $460000$ elements, but the exact number of generators was not computed in \cite{Gub}. Later Obraztsov \cite{Obr} showed that one can ensure that $\d(S) = 3$ in Guba's construction.

Dealing with $2$-generated subgroups is essential in Guba's approach and it seems that his method fails to produce finitely generated simple groups which are not $n$-generated already for $n = 3$. Thus the question of whether there exists a simple group $S$ with $\d(S) = n$ was open for $n>3$. Theorem \ref{main1} together with (\ref{triv}) immediately implies the following.

\begin{cor}
For every positive integer $n$ there exists a simple group $Q$ with $\d(Q)=n$.
\end{cor}

As we already mentioned, the proof of Theorem \ref{main1} explores coprime torsion. To the best of our knowledge, such a torsion is the only source of counterexamples to the inequality (\ref{triv}). Thus we put forward the following conjecture, which is obviously a stronger version of (\ref{triv}).

\begin{conj} \label{conj1}
Let $G$ be a torsion free discrete group. Then $\beta_1^{(2)}(G) \leq \nrk(G)-1.$
\end{conj}

If Conjecture \ref{conj1} held, it would have immediate applications to some outstanding problems (see the discussion in the last section of our paper). Our next result provides an evidence towards the conjecture.

\begin{thm} \label{main2}
Let $G$ be a finitely generated group. If $G$ is a limit of left orderable amenable groups in the space of marked group presentations, then
\begin{equation*} \label{maineq}
\beta_{1}^{(2)}(G) \leq \beta_1(G)-1 \leq \nrk(G)-1.
\end{equation*}
\end{thm}

It is worth noting that, in general, one cannot expect that $\beta_1^{(2)}(G) \leq \beta_1(G) -1$ even among left orderable groups, see Remark \ref{rem}.

Theorem \ref{main2} has an application to the approximation of the first $\ell^2$-Betti number by its finite dimensional analogues. Let $G$ be a finitely generated residually finite group and let $\{ N_i\}_{i\in \mathbb N}$ be a sequence of finite index normal subgroups of $G$ such that
\begin{equation}\label{N}
N_1\ge N_2 \ge \ldots ,\;\;\; {\rm and}\;\;\; \bigcap\limits_{i=1}^\infty N_i=\{ 1\} .
\end{equation}
The L\"uck's Approximation Theorem \cite{Lueck94} implies that if $G$ is finitely presented, then
\begin{equation}\label{approx}
\beta_1^{(2)} (G)= \lim \limits_{i\to \infty} \frac{\beta_1 (N_i)}{[G:N_i]}.
\end{equation}
where $\beta_1^{(2)} (G)$ is the first $\ell^2$-Betti number of $G$. In particular the limit in (\ref{approx}) exists and is independent of the choice of the sequence of finite index normal subgroups satisfying (\ref{N}). Moreover, in \cite{LO} L\"uck and the first author proved that the inequality
\begin{equation}\label{ineq}
\beta_1^{(2)} (G)\ge \lim \limits_{i\to \infty}\sup \frac{\beta_1 (N_i)}{[G:N_i]}.
\end{equation}
holds even if $G$ is not finitely presented. On the other hand, for infinitely presented groups the inequality in (\ref{ineq}) can be strict; indeed there exist groups for which the left side of (\ref{ineq}) is positive, while the right side equals zero \cite{LO}.

Using Theorem \ref{main2}, the inequality (\ref{ineq}), and multiplicativity of the first $\ell^2$-Betti number, we immediately obtain the following.
\begin{cor}
Let $G$ be a finitely generated group. If $G$ is a limit of left orderable amenable groups in the space of marked group presentations, then (\ref{approx}) holds for every sequence of finite index normal subgroups satisfying (\ref{N}).
\end{cor}

The corollary can be applied, for example, to subgroups of right angled Artin groups, which are residually torsion free nilpotent \cite{DK} and hence are limits left orderable amenable groups. Many of these subgroups are infinitely presented and thus the L\"uck theorem does not apply. Moreover, ``most" finitely generated subgroups are infinitely presented even in simplest right angled Artin groups, e.g., in direct products of free groups. Indeed by a theorem of Baumslag and Roseblade \cite{BR} a finitely presented subgroup of a direct product $F_m\times F_n$ is either free or virtually a direct product of finite index subgroups of the multiples.

\section{Preliminaries}
Given a word $W$ in an alphabet $\mathcal A$, we denote by $\| W\| $ its length. We also
write $W\equiv V$ to express the letter--for--letter equality of
words $W$ and $V$. For a group $G$ generated by a set $X\subseteq G$, we denote by $\Gamma (G,X)$ the corresponding Cayley graph. Given a (combinatorial) path $p$ in  $\Gamma (G,X)$, we denote by $\Lab (p)$ its label and by $p_-$ and $p_+$ its starting and ending point, respectively. The length of $p$ (i.e.,  the number of edges in
$p$) is denoted by $\ell (p)$. The {\it word length} $|g|_X$ of an element $g\in G$ is
defined to be the length of a shortest word in $X$
representing $g$ in $G$.

The metric space $\mathcal G_n$ of marked $n$-generated groups consists of pairs $(G,S)$, where $G$ is a group and $S$ is an ordered generating set of $G$ of cardinality $n$. Such pairs are in 1-to-1 correspondence with epimorphisms $F_n\to G$, where $F_n$ is the free group of rank $n$, and thus the set $\mathcal G_n$ can be identified with the set of all normal subgroups of $F_n$. Then the distance between $N_1, N_2\lhd F_n$ is then defined by
$$d(N_1, N_2) =
\inf \left\{2^{-k} \mid k \in \N \colon N_1 \cap B_{F_n}(k)  = N_2 \cap B_{F_n}(k)\right\},$$
where $B_{F_n}(k)$ denotes the closed ball of radius $k$ in $F_n$  with respect to a fixed basis.

It is well-known that $\mathcal G_n$ is a compact for every $n$. We say that a sequence of $n$-generated groups $(G_i)_{i\in \N}$ converges to an $n$-generated group $G$ if $(G_i, S_i)_{n\in \N}\subseteq \mathcal G_n$ converges to $(G,S)\in \mathcal G_n$ for some generating sets $S_i$ of $G_i$, $i\in \N$, and $S$ of $G$, respectively. Typical examples of convergent sequences arise from chains of normal subgroups. For instance, if $G_1\lhd G_2\lhd \cdots $ (respectively, $H_1\rhd H_2 \rhd \cdots $) is a chain of normal subgroups of $G$, then the sequence of groups
$(G/G_i)_{i \in \N}$ (respectively, $(G/H_i)_{i\in \N}$) converges to $G/\left(\bigcup_{i\in \N} G_i\right)$ (respectively, $G/\left(\bigcap_{i\in \N} H_i\right)$).

Recall that a group is hyperbolic if it admits a finite presentation with linear isoperimetric function. Similarly a group $G$ is hyperbolic relative to a collection of subgroups $\Hl $ if it admits a finite relative presentation with linear isoperimetric function. For details we refer to \cite{Gro}, \cite{Osi06a} and references therein.

Throughout the rest of this section let $G$ is be a group hyperbolic relative to a collection of subgroups $\Hl $. Let
$$
\mathcal H=\bigsqcup\limits_{\lambda\in \Lambda} (H_\lambda\setminus \{ 1\}).
$$
By the definition of relative hyperbolicity there exists a finite subset $X\subseteq G$, called a {\it relative generating set } of $G$,  such that $G$ is generated by $X\cup\mathcal H$. Let $\G $ be the Cayley graph of $G$ with respect to the generating set $X\sqcup \mathcal H$. We stress that all unions are disjoint, which means, for example, that if some element $h\in G$ belongs to $H_\lambda \setminus\{ 1\}$ and $H_\mu\setminus\{ 1\}$ for some $\mu\ne \lambda$, then there are (at least) two edges at each vertex of $\G $ corresponding to $h$: one is labelled by $h$, which we think of as an element of $H_\lambda \setminus\{ 1\}$ and the other is labelled by $h$, which we think of as an element of $H_\mu \setminus\{ 1\}$.

Let $q$ be a path in the Cayley graph $\G $. A (non--trivial)
subpath $p$ of $q$ is called an {\it $H_\lambda $-component} for some
$\lambda \in \Lambda $ if the label of $p$ is a word in the
alphabet $H_\lambda\setminus \{ 1\} $ and $p$ is maximal, i.e., is not contained in a bigger
subpath of $q$ labelled by a word in the
alphabet $H_\lambda\setminus \{ 1\} $.

Two $H_\lambda $-components $p_1, p_2$
of a path $q$ in $\G $ are called {\it connected} if there exists a
path $c$ in $\G $ that connects some vertex of $p_1$ to some vertex
of $p_2$ and $c$ is either trivial or ${\phi (c)}$ is a word consisting of letters from $
H_\lambda\setminus\{ 1\} $. In algebraic terms this means that all
vertices of $p_1$ and $p_2$ belong to the same coset $gH_\lambda $
for a certain $g\in G$. Note that we can always assume that $c$ has
length at most $1$, as every nontrivial element of $H_\lambda
\setminus\{ 1\} $ is included in the set of generators.  An
$H_\lambda $--component $p$ of a path $q$ is called {\it isolated }
in $q$ if no distinct $H_\lambda $--component of $q$ is connected
to $p$.

To every subset $\Omega$ of $G$, we can associate a length
$|\cdot |_\Omega $ on $G$ as follows. If $g\in \langle \Omega \rangle$, then $|g|_\Omega $ is the word length of $g$ with respect to
$\Omega $. Otherwise we set $|g|_\Omega =\infty $. Further for any path $p$ in
$\G $, we define its \emph{$\Omega $-length} as $\ell_\Omega (p)=|\Lab (p)|_\Omega$.

The lemma below was proved in \cite[Lemma 2.27]{Osi06a}.
\begin{lem}\label{Omega}
Let $G$ be a group that is hyperbolic relative to a collection of
subgroups $\Hl $. Then there exists a finite subset $\Omega
\subseteq G$ and a constant $L>0$ such that the following condition
holds. Let $q$ be a cycle in $\G $, $p_1, \ldots , p_k$ a set of
isolated components of $q$. Then the $\Omega $--lengths of $p_i$'s
satisfy
$$ \sum\limits_{i=1}^k l_\Omega (p_i)\le Ll(q).$$
\end{lem}

The next three algebraic results are well known. They can be found in \cite[Theorem 1.4]{Osi06a}, \cite[Theorem 2.40]{Osi06a}, and \cite[Theorem 4.3 and Corollary 1.7]{Osi06b}, respectively. Recall that an element $g\in G$ {\it loxodromic} if it has infinite order and is not conjugate to an element of one of the subgroups $H_\lambda $. Recall also that a group is {\it elementary} if it contains a cyclic subgroup of finite index.

\begin{lem}\label{Eg}
Suppose a group $G$ is hyperbolic relative to a collection of subgroups $\Hl$. Let $g$ be a
loxodromic element of $G$. Then the following conditions hold:
\begin{enumerate}
\item[(a)] There is a unique maximal elementary
subgroup $E_G(g)\le G$ containing $g$.

\item[(b)] $E_G(g)=\{ h\in G\mid \exists\, m\in \mathbb{N}~ \mbox{such that}~ h^{-1}g^mh=g^{\pm m}\} $.

\item[(c)] The group $G$ is hyperbolic relative to the collection
$\Hl\cup \{ E_G(g)\} $.
\end{enumerate}
\end{lem}

\begin{lem}\label{malnorm}
Let $G$ be a group hyperbolic relative to a collection of subgroups $\Hl $. Then for every $\lambda \in \Lambda $ and $g\in G\setminus H_\lambda $, we have $|H_\lambda \cap H_\lambda ^g|<\infty $.
\end{lem}

\begin{lem}\label{hyp}
Suppose that a group $G$ is hyperbolic relative to a finite collection of hyperbolic
subgroups. Then $G$ is hyperbolic itself.
\end{lem}

In Section 4 we will also need the following result proved in \cite{Osi07}. (An independent proof for torsion free groups can also be found in \cite{GM}.) It may be thought of as an algebraic generalization of the Thurston's hyperbolic Dehn surgery theorem (we refer to \cite{GM,Osi07} for details).

\begin{thm}\label{DF}
Let $G$ be group hyperbolic relative to a collection of subgroups $\Hl $. Then for every finite subset $\mathcal A\subseteq G$, there exists a finite subset $\mathcal F\subseteq G\setminus \{ 1\} $ such that for any collection of subgroups  $\mathcal N=\{ N_\lambda \mid\lambda \in \Lambda \} $ satisfying $N_\lambda \lhd H_\lambda $ and $N_\lambda \cap \mathcal F=\emptyset $ for all $\lambda \in \Lambda $, the following hold.
\begin{enumerate}
\item[(a)] Let $N=\ll\bigcup_{\lambda \in \Lambda} N_\lambda \rr$ be the normal closure of $\bigcup_{\lambda \in \Lambda} N_\lambda$ in $G$. Then for every $\lambda \in \Lambda$, the natural map $H_\lambda/N_\lambda \to G/N$ is injective (equivalently, $H_\lambda \cap N=N_\lambda $).
\item[(b)] $G/N$ is hyperbolic relative to $\{ H_\lambda /N_\lambda \mid \lambda \in \Lambda\} $.
\item[(c)] The natural homomorphism $G\to G/N$ is injective on $\mathcal A$.
\end{enumerate}
\end{thm}

\section{Special elements of hyperbolic groups}

Let $G$ be a relatively hyperbolic group. We call an element $g\in G$  {\it special} if it is loxodromic and $E_G(g)=\langle g\rangle $. If $G$ is an ordinary hyperbolic group, it can be thought of as hyperbolic relative to the trivial subgroup. Then the same definition applies. In this case `loxodromic' simply means `of infinite order'.

It is not hard to show using Lemma \ref{malnorm} that if $G$ contains a nontrivial finite normal subgroup, then it has no special elements. It turns out that the absence of nontrivial finite normal subgroups is also a sufficient condition for the existence of special elements in a non-elementary relatively hyperbolic group. Indeed the following is an immediate consequence of the combination of \cite[Corollary 4.5]{Osi06b} and \cite[Lemma 3.8]{AMO}. (For hyperbolic groups, this result was proved by Olshanskii in \cite{Ols93}.)

\begin{lem}\label{special-rh}
Let $G$ be a non-elementary group hyperbolic relative to a collection of proper subgroups. Suppose also that $G$ has no nontrivial finite normal subgroups. Then $G$ contains a special element.
\end{lem}

The main goal of this section is to prove Proposition \ref{coset} which provides us with special elements in cosets of normal subgroups. Although the proposition only concerns hyperbolic groups, its proof uses relative hyperbolicity. In fact, one could obtain the same result without mentioning relative hyperbolicity. This can be done, for example, by modifying the approach taken in  \cite{Ols93}. However we prefer to provide a proof based on relative hyperbolicity as it seems shorter; moreover, some technical results obtained in the course of the proof (e.g., the lemma below) may be useful elsewhere.

\begin{lem}\label{ahn}
Let $G$ be a relatively hyperbolic group, $h\in G$ a special element. Then for every $a\notin E_G(h)$ there exists a positive integer $n$ such that the element $g=ah^n$ is special.
\end{lem}

\begin{proof}
Suppose that $G$ is hyperbolic relative to $\Hl $. By Lemma \ref{Eg}, the group $G$ is also hyperbolic relative to $\Hl \cup \{ K\} $, where $K=E_G(h)$. The fact that $ah^n$ is loxodromic for all but finitely many $n$ is proved in \cite[Lemma 4.4]{Osi06b}. It remains to show that $E_G(g)=\langle g\rangle $.

Let $X$ be a finite relative generating set of $G$ with respect to $\Hl \cup \{ K\} $. Without loss of generality we can assume $a\in X$. Let also $L$ be the constant given by Lemma \ref{Omega} applied to the collection of peripheral subgroups $\Hl \cup \{ K\} $. Since $|\Omega |<\infty $, there exists $n\in \mathbb N$ such that
\begin{equation}\label{hn}
|h^n|_\Omega \ge 34L.
\end{equation}
Let $$\mathcal H =\left(\bigsqcup\limits_{\lambda\in \Lambda} H_{\lambda}\setminus \{ 1\} \right)\sqcup (K\setminus \{ 1\} ).$$ We denote by $W$ the word $ah^n$ in $X\cup \mathcal H$, where $h^n$ is considered as single letter from $K\setminus \{ 1\}$. Thus $\| W\| =2$. Let $p$ be a path in $\G $ such that $\Lab (p)\equiv W^m$ for some $m\in \mathbb N$. Then $p=r_1p_1\ldots r_{2m}p_{2m}$, where $p_i$'s are edges labelled by $h^{\pm n}$ and $r_i$'s are edges labelled by $a$. Clearly $p_i$'s are $K$-components of $p$ and no two consecutive components of $p$ are connected as $a\notin K$.

We first show that all $K$-components of $p$ are isolated. Indeed suppose that $p_i$ and $p_j$ are connected for some $j>i$ and $j-i$ is minimal possible. Note that  $j=i+1+k$ for some $k\ge 1$, as no two consecutive $K$-components of $p$ are connected. Let $t$ denote the segment of $p$ with $t_-=(p_i)_+$ and $t_+=(p_j)_-$, and let $e$ be an empty path or an edge in $\G $ labelled by an element of $K \setminus\{ 1\}$ such that $e_-=(p_i)_+$, $e_+=(p_j)_-$. Note that the components $p_{i+1}, \ldots , p_{i+k}$ are isolated in the cycle $te^{-1}$ as otherwise we can pass to another pair of connected components with smaller value of $j-i$. By Lemma \ref{Omega} we have
$$
\sum\limits_{s=1}^k \lo (p_{i+s})\le  L\ell (te^{-1})\le L (2k+2).
$$
Hence $\lo (p_{i+s})\le L (2+2/k)\le 4L$ for some $s$, which
contradicts (\ref{hn}). Thus all $K$-components of $p$ are isolated.

Further let $d\in E_G(g)$. Then $d$ commutes with arbitrary large powers of $g$. We choose $k$ such that $[d,g^{8k}]=1$ and
\begin{equation}\label{k}
k \ge |d|_{X\cup\mathcal H}.
\end{equation}
Let $$c=upv^{-1}q^{-1}$$ be a quadrangle in $\G $ such that $u,v$ are geodesics whose labels represent $d$ in $G$, and $\Lab (p)\equiv \Lab (q)\equiv W^{8k}$. Note that $\ell (p)=\ell (q)=16k$. Hence by (\ref{k}) we have
\begin{equation}\label{lc}
\ell (c)\le 34k.
\end{equation}
Let $\{ p_i\mid i\in I\}$ (respectively, $\{ p_i\mid i\in J\}$) be the set of all $K$-components of $p$ which are isolated (respectively, not isolated) in $c$. By Lemma \ref{Omega} we obtain
$$
\sum\limits_{i\in I} \lo (p_i)\le 34kL.
$$
Since $\lo (p_i)\ge  34L$ for every $i\in I$ by (\ref{hn}), we have $|I|\le k$. Since the total number of $K$-components of $p$ is $8k$, we have $|J|\ge 7k$. Recall that no two distinct components of $p$ are connected. Further, by (\ref{k}) at most $2k$ components of $p$ are connected to components of $u$ or $v$. Hence there exists a subset $J_0\subseteq J$ such that $|J_0|\ge 5k$, and $p_i$ is not connected to a component of $u$ or $v$ whenever $i\in J_0$. Thus for every $i\in J_0$, $p_i$ is connected to a component of $q$.

Let us fix any $i\in J_0$. Let $q_j$ be the $K$-component of $q$ connected to $p_i$. In what follows we assume that $i\le 4k$ (the case $i\ge 4k$ is symmetric). Let $f_1$ be the empty path or an edge of $\G $ connecting $(q_j)_+$ to $(p_i)_+$ and labelled by an element of $K$. Let $$c_1=f_1[(p_i)_+, p_+]v^{-1}[(q_j)_+, q_+]^{-1},$$ where $[(p_i)_+, p_+]$ and $[(q_j)_+, q_+]$ denote the corresponding segments of $p$ and $q$. By our assumption, $[(p_i)_+, p_+]$ has at least $4k$ $K$-components. Repeating the arguments from the previous paragraph, we can show that at least $k$ of them are connected to components of $[(q_j)_+, q_+]$. In particular, there exists a $K$-component $p_{i+s}$ of $p$ connected to a $K$-component of $q$ for some $s\in \mathbb N$. Assume that $s$ is minimal possible. Let $q_{j+t}$ be the $K$-component of $q$ connected to $p_{i+s}$ and let $f_2$ be the empty path of an edge of $\G$ connecting $(p_{i+s})_-$ to $(q_{j+t})_-$ (see Fig. 1).

\begin{figure}
    \def\svgwidth{400pt}
  \centering{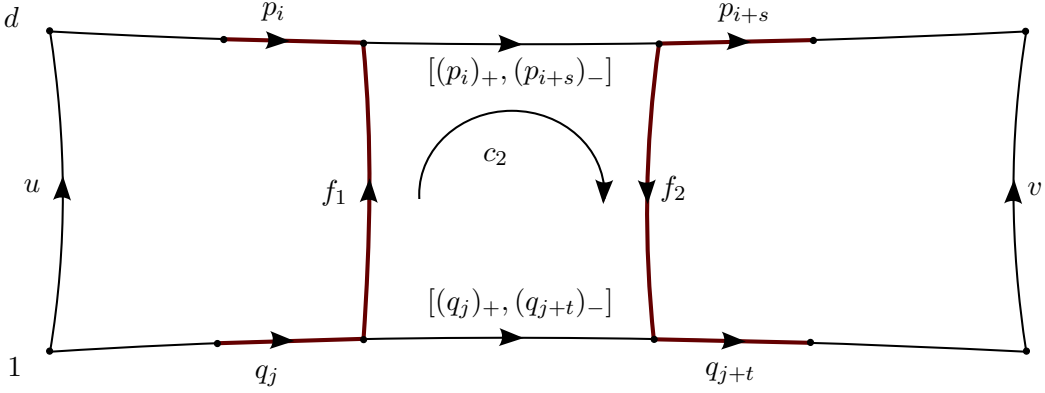\\
  \caption{Decomposition of the cycle $c$ in the proof of Lemma \ref{special-rh}. } }
\end{figure}

Using minimality of $s$ and the fact that all components of $p$ (respectively, $q$) are isolated in $p$ (respectively, $q$), we obtain that all components from the set $$M=\{ p_l\mid i<l<i+s\} \cup \{ q_l\mid j<l<j+t\} $$ (if any) are isolated in the loop $$c_2=f_1[(p_i)_+,(p_{i+s})_-]f_2[(q_{j})_+, (q_{j+t})_-]^{-1},$$ where $[(p_i)_+,(p_{i+s})_-]$ and $[(q_{j})_+, (q_{j+t})_-]$ are the corresponding segments of $p$ and $q$, respectively. Note that $\ell (c_2) \le 2|M|+4\le 6|M|$ if $M\ne \emptyset $. However this contradicts (\ref{hn}) and Lemma \ref{Omega} applied to $c_2$ and components from $M$. Hence $M=\emptyset $, i.e., $s=t=1$. Now reading the label of $c_2$ we obtain $h_1ah_2a^{-1}=1$, where  $h_1, h_2$ are elements of $K$ represented by labels of $f_1$ and $f_2$, respectively. Since $K$ is infinite cyclic, Lemma \ref{malnorm} implies $h_1=h_2=1$. Thus $f_1$ (as well as $f_2$) is the empty path. This implies, via reading the label of the cycle $u[p_-,(p_i)_+][q_-, (q_j)_+]^{-1}$, that $d$ equals a power of $g$. Thus $E_G(g)=\langle g\rangle $.
\end{proof}

Let $G$ be a hyperbolic group, $H\le G$. Denote by $H^0$ the set of all elements of $H$ of infinite order. Let
\begin{equation}\label{EH}
E_G(H)=\bigcap\limits_{h\in H^0} E_G(h).
\end{equation}
It is easy to show that if $H$ is non-elementary, then $E_G(H)$ is finite \cite{Ols93}.

The following lemma was proved by Olshanskii \cite{Ols93}.

\begin{lem}\label{Ols}
Let $G$ be a hyperbolic group, $H\le G$ a non-elementary subgroup.  Then there exists an element $h\in H$ of infinite order such that $E_G(h)=\langle h\rangle \times E_G(H)$.
\end{lem}

We are now ready to state and prove the main result of this section.

\begin{prop}\label{coset}
Let $G$ be a hyperbolic group without nontrivial finite normal subgroups. Then for every nontrivial element $a\in G$ and every $x\in G$, there exists a special element $g\in x\ll a\rr^G $.
\end{prop}
\begin{proof}
Let $H=\ll a\rr^G$. Note that since $G$ is non-elementary, so is $H$. Indeed otherwise $H$ contains an infinite cyclic characteristic subgroup $C$, which is normal in $G$. Then the centralizer $C_G(C)$ has finite index in $G$. However, $C$ has finite index in $C_G(C)$ (say, by Lemma \ref{Eg}) and hence $G$ is elementary. A contradiction.

Thus $E_G(H)$ is a finite subgroup of $G$. Since $H$ is normal in $G$, it easily follows from (\ref{EH}) that $E_G(H)$ is normal in $G$ as well. Therefore $E_G(H)=\{ 1\} $. By Lemma \ref{Ols} there exists $h\in H$ of infinite order such that $E_G(h)=\langle h\rangle$, i.e., $h$ is a special element of $G$. If $x\in \langle h\rangle \le H$, then we can take $g=h$. Otherwise we apply Lemma \ref{ahn} and we obtain a special element $g\in xH$.
\end{proof}

\section{Simple groups with positive $\ell^2$-Betti numbers}

In this section we deal with groups admitting presentations of some special kind. Namely, we will call an {\it irreducible torsion presentation} any presentation $\mathcal P$ of the form
\begin{equation}\label{irred}
\langle X \mid  R_1^{n_1}, \dots, R_k^{n_k} \rangle,
\end{equation}
where the order of the element of $G$ represented by $R_i$ is $n_i$ for every $i=1, \ldots , k$. Associated to this presentation is the quantity $$\sigma (\mathcal P)= \sum_{i=1}^k \frac{1}{n_i}.$$
The next result is a simplification of a theorem of Peterson and the second author \cite{PT}.

\begin{thm}[Theorem 3.2 in \cite{PT}]  \label{PT} Let $G$ be a group given by an irreducible torsion presentation (\ref{irred}), where $|X|<\infty $. Then
$$\beta_{1}^{(2)}(G) \geq |X|-1 - \sigma (\mathcal P). $$
\end{thm}

\begin{proof}[Proof of Theorem \ref{main1}]
Let us fix $n\in \mathbb N$ and let $X=\{ x_1, \ldots , x_n\} $ and
$$
G_0=\langle X \mid x_1^p=1, \, \ldots , \, x_n^p=1\rangle ,
$$
where $p$ is a prime satisfying $n/p<\e $. In particular, $\beta _1^{(2)} (G_0)> n-1-\e $. We enumerate all pairs of elements $$X\times (G_0\setminus \{1\})=\{ (x_{m_1},g_1),(x_{m_2},g_2), \ldots \} $$ and construct a sequence of groups and epimorphisms $$G_0\stackrel{\alpha_0}\longrightarrow G_1\stackrel{\alpha_1}\longrightarrow \ldots $$ as follows. Below we use the same symbols to denote elements of $G_0$ and their images in $G_1, \ldots $. At step $i$ we assume that the group $G_i$ is already constructed and satisfies the following conditions.
\begin{enumerate}
\item[(a)] $G_i$ is a non-elementary hyperbolic group without finite normal subgroups.
\item[(b)]  $G_i$ has an irreducible torsion presentation $\mathcal P_i$  with $\sigma (\mathcal P_i)<\e $.
\item[(c)] If $i\ge 1$, then for every $j=1, \ldots ,   i$, we either have $g_j=1$ in $G_i$ or $x_{m_j}\in \ll g_j\rr^{G_i}$.
\end{enumerate}

It is straightforward to verify (a)--(c) for $i=0$. The group $G_{i+1}$ is obtained from $G_i$ as follows. If $g_{i+1}=1$ in $G_{i}$, then we set $G_{i+1}=G_i$ and $\alpha _{i}=id$. Otherwise by (a) and Proposition \ref{coset}, there exists a special element $h_i\in x_{m_{i+1}}\ll g_{i+1}\rr^{G_i}$. By Lemma \ref{Eg}, $G_i$ is hyperbolic relative to $H_1=\langle h_i\rangle $. Further by Lemma \ref{special-rh}, there exists an element $t_i\in G_i$ that is special in $G_i$ considered as a group hyperbolic relative to $H_1$. Using Lemma \ref{Eg} again we conclude that $G_i$ is hyperbolic relative to $\{ H_1, H_2\}$, where $H_2=\langle t_i\rangle $.

Let (\ref{irred}) be the irreducible torsion presentation $\mathcal P_i$ of $G_i$ and let $\mathcal A$ be the set of all powers of  elements of $G_i$ represented by $R_1, \ldots , R_k$. Obviously $\mathcal A$ is finite. Let $\mathcal F$ be the finite set given by Theorem \ref{DF} applied to $G_i$ and the collection of parabolic subgroups $\{ H_1, H_2\}$. We choose a prime $q_i>p$ such that the subgroup $N_1=\langle h_i^{q_i}\rangle $ does not contain elements of $\mathcal F$ and $\sigma (\mathcal P_i)+1/q_i<\e $.  Let $G_{i+1}$ be the group given by the presentation $$\mathcal P_{i+1}=\langle \mathcal P_i\mid h_i^{q_i}=1\rangle .$$

By Theorem \ref{DF} applied to the subgroups $N_1\lhd H_1$ and $N_2=\{ 1\}\lhd H_2$, we obtain that $H_1/N_1\cong \mathbb Z/q_i\mathbb Z$ and $H_2$ naturally embed in $G_{i+1}$, and $G_{i+1}$ is hyperbolic relative to $\{ H_1/N_1, H_2\}$. In particular, $G_{i+1}$ is hyperbolic by Lemma \ref{hyp}. If $G_{i+1}$ is elementary, then the infinite cyclic subgroup $H_2$ has finite index in $G_{i+1}$. Hence there is a finite index subgroup $C\le H_2$ which is normal in $G$. In particular, $C^{h_i}\cap C=C$ and hence $h_i\in C\le H_2$ by Lemma \ref{malnorm}. However this is impossible since $h_i$ is nontrivial and has finite order in $G_{i+1}$. Thus $G_{i+1}$ is non-elementary. Further suppose that $K$ is a finite normal subgroup of $G_{i+1}$. Then $K$ is centralized by a finite index subgroup of $G_{i+1}$. In particular, $K$ is centralized by a nontrivial element of $H_2$. Applying Lemma \ref{malnorm} again, we conclude that $K\le H_2$, which implies $K=\{ 1\}$. Thus part (a) of the inductive assumption holds for $G_{i+1}$.

To verify (b) it suffices to note that $\mathcal P_{i+1}$ is an irreducible torsion presentation by Theorem \ref{DF}. Indeed irreducibility is ensured by the choice of the set $\mathcal A$, part (c) of Theorem \ref{DF}, and the fact that $H_1/N_1$ embeds in $G_{i+1}$.

Finally we note that adding relation $g_{i+1}=1$ to $G_{i+1}$ yields $h_i=x_{m_{i+1}}$ and hence $x_{m_{i+1}}^{q_i}=1$ in $G_{i+1}/\ll g_{i+1}\rr^{G_{i+1}}$. However we also have $x_{m_{i+1}}^p=1$ in $G_0$ (and hence in $G_{i+1}/\ll g_{i+1}\rr^{G_{i+1}}$). Since $q_i$ and $p$ are distinct primes, we have $x_{m_{i+1}}=1$ in $G_{i+1}/\ll g_{i+1}\rr^{G_{i+1}}$, i.e., $x_{m_{i+1}}\in \ll g_{i+1}\rr^{G_{i+1}}$. Together with the inductive assumption this implies (c) for $G_{i+1}$. Thus the inductive step is completed.

Let now $Q$ be the limit group. That is, let $Q=G_0/\bigcup\limits_{i=1}^\infty {\rm Ker} (\alpha _i\circ \cdots \circ \alpha _0) $. Note that $Q$ is the limit of the groups $G_i$ in the topology of marked group presentations and hence $$\beta_1^{(2)} (Q)\ge \limsup_{i\to \infty} \beta_1^{(2)} (G_i)\ge n-1-\e .$$ Here the first inequality follows from semicontinuity of the first $\ell^2$-Betti number \cite{Pich} and the second one follows from (b) and Theorem \ref{PT}. Note also that $Q$ is infinite by the following standard argument. If $Q$ was finite, it would be finitely presented and hence all relations of $Q$ would hold in some $G_i$, which contradicts (a). Finally we note that $Q$ is simple. Indeed, let $q\in Q$ be a nontrivial element. Then part (c) of the inductive assumption ensures that $x_1, \ldots , x_n\in \ll q\rr ^Q$. Hence $Q=\ll q\rr ^Q$.
\end{proof}

\begin{rem}\label{hypgroups}
Observe that the hyperbolic groups $G_i$ constructed in the course of proving Theorem \ref{main1} have normal rank $1$ for all $i$ large enough. This provides examples of hyperbolic groups with normal rank $1$ and arbitrary large first $\ell^2$-Betti number. It is not difficult to construct explicit examples of this type. In fact, using similar methods the reader can verify that for every integer $n\ge 2$ and any sufficiently large distinct primes $p_1, \ldots, p_{n}$, the group
$$
H_n=\langle x_1, \ldots, x_n\mid x_1^{p_1}=\ldots = x_n^{p_1}=1, \, (x_1x_2)^{p_2}=\ldots =(x_1x_n)^{p_n}=1\rangle
$$
is hyperbolic, has $\beta_1^{(2)} (H_n)> n-2$, and $\nrk (H_n)=1$ as $H_n=\ll x_1\rr$.
\end{rem}

\section{Normal rank of limits of left orderable amenable groups}

The main goal of this section is to proof Theorem \ref{main2}.

\begin{proof}
The second inequality is obvious. Let us prove the first inequality. In order to estimate the first $\ell^2$-Betti number we use its relation with the Murray-von Neumann dimension of the space of $\ell^2$-cocycles
$Z^1(G,\ell^2 G) := \{ c\colon G \to \ell^2 G \mid c(gh) = \lambda(g) c(h) + c(g) \}$, i.e.\
$$\beta_{1}^{(2)}(G) - \beta_0^{(2)}(G) + 1 = \dim_G Z^1(G,\ell^2 G).$$
This identification was obtained in \cite{PT}. See also $\cite{PT}$ for further information and references concerning the dimension function.

We may assume that $G$ is infinite and hence $\beta_0^{(2)}(G)=0$.
For any finitely generated group $G$, we may consider the first part of the standard resolution:
$$\cdots \to \Z G^{\oplus r} \stackrel{A}{\to} \Z G^{\oplus n} \to \Z G \to \Z \to 0,$$
(with $r \in \N \cup \{\infty\}$) associated with a presentation
$G = \langle h_1,\dots,h_n \mid q_1,q_2,\dots \rangle.$
Note that since the abelianization $G_{ab}$ is finitely generated and hence finitely presented, there exists $r' \in \N$, such that the natural map
$$\langle h_1,\dots,h_n \mid q_1,q_2,\dots,q_{r'} \rangle_{ab} \to G_{ab}$$
is an isomorphism.

Clearly, any $1$-cocycle $c\colon G \to \ell^2 G$ is determined by its values on the generators of $G$, so that we get a natural injective map
$$ {\rm ev} \colon Z^1(G,\ell^2 G) \to \ell^2 G^{\oplus n}.$$
Moreover, a set of values $\xi \in \ell^2 G^{\oplus n}$ lies in the image of this map if and only if $A^* \xi =0$, where $A^* \in M_{r,n} \Z[G]$ is the formal adjoint of the matrix $A$ appearing in the resolution above. If $\xi \in \im ({\rm ev})$, we denote the associated cocycle by $c_{\xi} \colon G \to \ell^2 G$.

Let $k$ be an integer and $g_1,\dots,g_k$ are elements of $G$. Then, we can extend $A^*$ to a matrix $B \in M_{r+k,n} \Z[G]$ by adding the Fox derivatives of $g_1,\dots,g_k$. It is easy to check that $B \xi = 0$ for $\xi \in \ell^2 G^{\oplus n}$ if and only if the associated cocycle $c_{\xi} \colon G \to \ell^2 G$ vanishes on $g_1,\dots,g_k$.

Consider the augmentation homomorphism $\epsilon \colon \Z[G] \to \Z$. The kernel of $\epsilon(A^*) \colon \Z^{\oplus n} \to \Z^{\oplus r}$ consists precisely of the admissable values of a group homomorphism $c \colon G \to \Z$ on the chosen set of generators of $G$. Similarly, the kernel of $\epsilon(B) \colon \Z^{\oplus n} \to \Z^{\oplus r+k}$ describes the set of homomorphisms from $G$ to $\Z$ which vanish on $g_1,\dots,g_k$. Thus, setting $k:=\beta_{1}(G)={\rm rk}_{\Z} H^1(G,\Z) = {\rm rk}_{\Z}\hom(G,\Z)$, we find $g_1,\dots,g_k \in G$, such that $\epsilon(B)$ is injective.

We also consider the restricted matrix $B' \in M_{r'+k,n} \Z[G]$ and note that
$\ker(B' \colon \ell^2 G^{\oplus n} \to \ell^2 G^{\oplus r'+k}) \supset \ker(B \colon \ell^2 G^{\oplus n} \to \ell^2 G^{\oplus r+k})$
whereas
$$\ker(\epsilon(B') \colon \Z^{\oplus n} \to \Z^{\oplus r'+k}) = \ker(\epsilon(B) \colon \Z^{\oplus n} \to \Z^{\oplus r+k}) = \{0\}$$
by our choice of $r'$.

We want to argue that $B' \colon \ell^2 G^{\oplus n} \to \ell^2G^{\oplus r'+k}$ is injective. This will finish the proof. Indeed - assuming injectivity - it is then easy to see that
$$\beta_1^{(2)} (G)+1 = \dim_G Z^1(G,\ell^2 G) = \dim_G \ker(A^*) \leq k.$$
The last inequality follows, since $B'$ defines an injection from $\ker(A^*)$ into $\ell^2 G^{\oplus k}$ and the Murray-von Neumann dimension is monotone.

Assume now that $G$ is as above and let $(G_i)_{i \in \N}$ be a sequence of left-orderable amenable groups which converges to $G$ in the space $\mathcal G_n$ of marked groups.
First of all, we pick a lift of $B'$ in $M_{r'+k,n}\Z[F_n]$ and consider its images $B'_i \in M_{r'+k,n} \Z[G_i]$.

By Morris' result \cite{MR2286034}, every left orderable amenable group is locally indicable. Now, this implies that $B'_i \colon \Z[G_i]^{\oplus n} \to \Z[G_i]^{\oplus r'+n}$ is injective. Indeed, every locally indicable group is conservative \cite[Corollary 4.5]{MR704287} by a result of Gersten, and for conservative groups, injectivity of a matrix $B'_i$ over $\Z[G_i]$ is implied by injectivity of the augmented matrix $\epsilon(B'_i)$. In our case $\epsilon(B'_i)$ is identified with the matrix $\epsilon(B)$, which is injective.

The following proposition is an immediate consequence of Elek's results in \cite{MR1952401}, compare also \cite[Theorem 6.37]{MR1926649}.

\begin{prop}[Elek] Let $\Gamma$ be an amenable group and $A \in M_{p,q} \C[\Gamma]$ be a matrix. If there exists $\xi \in \ell^2 \Gamma^{\oplus q}$, such that $A\xi =0$, then there exists $\xi' \in \C \Gamma^{\oplus q}$, such that $A \xi'=0$.
\end{prop}

The proposition, applied to $\Gamma=G_i$ and $A=B'_i$, yields that also the induced map on the $\ell^2$-level
$B'_i \colon \ell^2 (G_i)^{\oplus n} \to \ell^2 (G_i)^{\oplus r'+k}$ is injective if and only if
$B'_i \colon \C[G_i]^{\oplus n} \to \C[G_i]^{\oplus r'+k}$ is injective. However, this happens if and only if $B'_i \colon \Z[G_i]^{\oplus n} \to \Z[G_i]^{\oplus r'+k}$ is injective, since $\Z \subset \C$ is a flat ring extension. Combining this with the result above, we conclude that the map
$$B'_i \colon \ell^2 (G_i)^{\oplus n} \to \ell^2 (G_i)^{\oplus r'+k}$$
is injective for all $i \in \N$.

We want to argue that this implies that the map $B' \colon \ell^2 G^{\oplus n} \to \ell^2 G^{\oplus r'+k}$ is injective too. This follows from the following variation of L\"uck's Approximation Theorem.

\begin{lem}
Let $(G_i)_{i \in \N}$ be a sequence of marked sofic groups converging to $G$ in $\mathcal G_n$. Let $A \in M_{p,q}(\Z[F_n])$ be a matrix. Then,
$$\lim_{i \to \infty} \dim_{G_i} \ker(A_i \colon \ell^2 G_i^{\oplus q} \to \ell^2 G_i^{\oplus p}) = \dim_G \ker(A' \colon \ell^2 G^{\oplus q} \to \ell^2 G^{\oplus p}),$$
where $A'$ denotes the image of $A$ in $M_{p,q}(\Z[G])$, and $A_i$ denotes the corresponding image in $M_{p,q}(\Z[G_i])$.
\end{lem}
\begin{proof}[Sketch of proof:]
First of all, upon replacing $A$ by $A^*A$, we may assume that $A$ is positive. For simplicity, let us assume that $A \in \Z[F_n]$. Denote the canonical trace on $\Z[G]$ by $\tau$ and the traces on $\Z[G_i]$ by $\tau_i$. To $A'$, we can associate a spectral measure $\mu_{A'}$ on $[0,\infty)$, which satisfies
$$\tau(A'^n) = \int_{\R} t^n d \mu_{A'}(t)$$
and define $\mu_{A_i}$ similarly using the traces $\tau_i$.
It is a standard fact that $\mu_{A_i}(\{0\}) = \dim_{G_i} \ker(A_i \colon \ell^2 G_i \to \ell^2 G_i)$. Moreover, since $(G_i)_{i \in \N}$ converges to $G$ it is easy to see that the spectral measures of $A_i$ converge in moments to the spectral measure of $A'$, i.e.\ $\int_{\R} t^n d \mu_{A_i}(t) \to \int_{\R} t^n d \mu_{A'}(t)$, for all $n \in \N$. Since the support of $\mu_{A_i}$ is bounded by $\|A\|_1$, this implies that $\mu_{A_i} \to \mu_{A'}$ weakly. We can conclude that
$\mu_{A'}(\{0\}) \geq \limsup_{i \to \infty} \mu_{A_i}(\{0\}).$

Now, since all $G_i$ are sofic, there is a universal constant $C$, only depending on $A \in \Z[F_n]$, such that
$$\mu_{A_i}([0, \varepsilon)) \leq \mu_{A_i}(\{0\}) + C \cdot |\log(\varepsilon) |^{-1}.$$
This follows by approximating the spectral measure of $A_i$ by spectral measures of integer matrices, and the corresponding fact for eigenvalue distributions of integer matrices, see \cite{sofapp}. We conclude
$$\mu_{A'}(\{0\}) \leq  \mu_{A'}([0, \varepsilon)) \leq \liminf_{i \to \infty}\mu_{A_i}([0, \varepsilon)) \leq \liminf_{i \to \infty} \mu_{A_i}(\{0\}) + C \cdot |\log(\varepsilon) |^{-1}.$$
Since this holds for all $\varepsilon>0$, we get $\lim_{i \to \infty} \mu_{A_i}(\{0\}) = \mu_{A'}(\{0\})$ as required. This finishes the proof. \end{proof}

Note that amenable groups are sofic, and we can conclude from the lemma and the vanishing of the kernels of $B'_i \colon \ell^2 (G_i)^{\oplus n} \to \ell^2 (G_i)^{\oplus r'+k}$ that the kernel of $B' \colon \ell^2 G^{\oplus n} \to \ell^2 G^{\oplus r'+k}$ is trivial. As we argued before, this shows
$\beta_1^{(2)}(G) \leq \beta_{1}(G)-1$
and finishes the proof.
\end{proof}

\begin{rem}\label{rem}
In general, even among left orderable groups, one cannot expect that
$$\beta_1^{(2)}(G) \leq \beta_1(G) -1$$
will always hold.
Indeed, Kropholler and Thurston \cite[Section 6]{berg} pointed out an explicit example of a finitely generated, left-orderable perfect group
$$G =\langle x, y, z \mid x^2=y^3=z^7=xyz \rangle$$
with vanishing first $\ell^2$-Betti number.
Clearly, such a group cannot be locally indicable. The group $G$ is the fundamental group of a Seifert fiber space, i.e.\ it is the fundamental group of a $S^1$-bundle over some $2$-dimensional orbifold.
More specifically, removing $x=yz$ from the generating set, we get: $$G= \langle y,z \mid zyzy^{-2}, yzyz^{-6} \rangle.$$ It is easily seen from this representation that $G$ is perfect. Indeed, the first differential in the chain complex which computes $H_*(G,\Z)$ is built from the total degrees of the variables $y$ and $z$ in the relations, and hence looks like $$d_1= \left( \begin{array}{cc} -1 & 2 \\  2 & -5\end{array} \right).$$
Since $\det(d_1) = 5 -4 =1$, we conclude that $H_1(G,\Z)= G/[G,G]=0$.

Moreover, $\beta_1^{(2)}(G)=0$, since $G$ has an infinite amenable normal subgroup. Taking free products of $G$ with itself, we can make the difference $\beta_1^{(2)}(G) - \beta_1(G)$ as large as we wish.
\end{rem}

\section{Some further questions}

Even in the presence of torsion it seemed hard to find examples of groups $G$, where adding a single relation forces the first $\ell^2$-Betti number to drop by large amounts. It was asked explicitly by Gromov in Section 8.A$_4$ of \cite{asy}, whether it can drop by more than one. There are easy counterexamples to this question like $PSL(2,\Z) \ast \Z = \langle a,b,c \mid a^2 = b^3 = 1 \rangle$, where the relation $ab$ forces a drop of $7/6$. However, we are not aware of a classical construction of groups were one additional relation forces a drop of the first $\ell^2$-Betti number by two or more. Theorem \ref{main1} gives a negative answer to Gromov's question in a very strong sense. Indeed, adding any non-trivial relation to a simple group forces a drop of the first $\ell^2$-Betti number to zero. Again, since existence of coprime torsion is essential to the construction, we conjecture:

\begin{conj} \label{conj2}
Let $G$ be a torsion free discrete group (or a $p$-group for a prime $p$) and let $g_1,\dots,g_n$ be arbitrary elements of $G$. Then $$\beta_1^{(2)}(G/ \! \ll g_1,\dots,g_n \rr) \geq \beta_1^{(2)}(G)-n.$$
\end{conj}

We want to note that a positive solution to Conjecture \ref{conj1} or Conjecture \ref{conj2} would have immediate applications to longstanding problems. For instance, an old problem of Wiegold asks whether every finitely generated perfect group has normal rank equal to one. The answer is positive for finite groups, but the question is wide open in the general case.  For any torsion free perfect group $G$, the group $G \ast G$ would yield a counterexample if either of our conjectures held.

Another well-known open problem is the Kervaire Conjecture, which can be restated as follows: For every nontrivial group $G$, the normal rank of $G\ast \mathbb Z$ is at least $2$. A positive answer is known for sofic groups, torsion free-groups, and some other particular classes. Clearly Conjecture \ref{conj1} or Conjecture \ref{conj2} would yield an alternative proof of the Kervaire conjecture for torsion free groups as $\beta_1^{(2)}(G\ast  \Z)>0$ for any infinite group $G$.

Yet another example is the Whitehead Asphericity Conjecture asserting that any subcomplex of a two-dimensional aspherical CW-complex is aspherical. An important case is given by $X \subset X \cup_f D^2$ for some $f \colon S^1 \to X$, with $X \cup_f D^2$ two-dimensional and contractible. While it is easy to see that $\nrk(\pi_1(X))=1$, Berrick and Hillman \cite{berhil} showed that $\beta^{(2)}_1(\pi_1(X))=0$ implies that $X$ is aspherical.

These examples obviously show that proving either of our conjectures in a ``nontrivial situation" should not be easy. In fact, the case of groups as in Theorem \ref{main2} can be thought of as ``trivial", although the proof is not elementary. Indeed the inequality (\ref{naive}) is derived from a stronger inequality $\beta_{1}^{(2)}(G) \leq \beta_1(G)-1$. This means that our theorem does not have any nontrivial applications to either of the above mentioned open problems.

Finally we remark that there are many other group invariants closely related to the first $\ell^2$-Betti number (e.g., cost, rank gradient, $p$-virtual deficiency, just to name a few). We refer to \cite{gaboriau,Lac,LO} for more details. To the best of our knowledge, suitable versions of Conjectures \ref{conj1} and \ref{conj2} remain open for these invariants.

\end{document}